\documentclass[10pt]{article}

\usepackage[english]{babel}
\usepackage[utf8x]{inputenc}
\usepackage[T1]{fontenc}

\usepackage[letterpaper,top=1in,bottom=1in,left=1in,right=1in,marginparwidth=1.75cm]{geometry}

\usepackage{amsmath,amsfonts,amsthm,amssymb,mathrsfs,dsfont} 
\usepackage{mathtools}
\usepackage{graphicx}
\usepackage[colorinlistoftodos]{todonotes}
\usepackage[colorlinks=true, allcolors=blue]{hyperref}
\usepackage[capitalize,nameinlink]{cleveref}
\usepackage{verbatim}
\usepackage{enumerate}
\usepackage{slantsc}

\newcommand{\R}{\mathbb{R}}

\newcommand{\CR}{\textsc{cr}}
\newcommand{\LCR}{\textsc{lcr}}
\newcommand{\I}{\mathcal{I}}
\newcommand{\g}{\mathcal{G}}

\global\long\def\R{\mathbb{R}}

\global\long\def\N{\mathbb{N}}

\global\long\def\E{\mathbb{E}}

\global\long\def\Pr{\textnormal{Pr}}

\newtheorem{theorem}{Theorem}[section]
\newtheorem*{namedtheorem}{\theoremname}
\newcommand{\theoremname}{testing}

\newtheorem{thm}[theorem]{Theorem}
\newtheorem{lemma}[theorem]{Lemma}

\newtheorem{prop}[theorem]{Proposition}

\newtheorem{cor}[theorem]{Corollary}

\newtheorem*{question*}{Question}

\theoremstyle{definition}
\newtheorem{definition}[theorem]{Definition}

\newtheorem{remark}[theorem]{Remark}
\newtheorem{rmk}[theorem]{Remark}

\theoremstyle{plain}

\title{On the $k$-planar local crossing number}
\author{John~Asplund \thanks{Dalton State College, Department of Technology and Mathematics,
jasplund@daltonstate.edu} \and Thao~Do \thanks{Massachusetts Institute of Technology, Department of Mathematics, thaodo@mit.edu} \and Arran~Hamm \thanks{Winthrop University, Department of Mathematics, hamma@winthrop.edu} \and 	Vishesh Jain \thanks{Massachusetts Institute of Technology, Department of Mathematics, visheshj@mit.edu}}
\usepackage{algpseudocode,algorithm,algorithmicx}
\date{}

\begin{document}
\maketitle

\begin{abstract}
Given a fixed positive integer $k$, the $k$-planar local crossing number of a graph $G$, denoted by $\LCR_k(G)$, is the minimum positive integer $L$ such that $G$ can be decomposed into $k$ subgraphs, each of which can be drawn in a plane such that no edge is crossed more than $L$ times. 
In this note, we show that under certain natural restrictions, the ratio $\LCR_k(G)/\LCR_1(G)$ 
is of order $1/k^2$, which is analogous to the result of Pach et al. \cite{PACH-k-planar} for the $k$-planar crossing number $\CR_k(G)$ (defined as the minimum positive integer $C$ for which there is a $k$-planar drawing of $G$ with $C$ total edge crossings). As a corollary of our proof we show that, under similar restrictions, one may obtain a $k$-planar drawing of $G$ with \emph{both} the total number of edge crossings as well as the maximum number of times any edge is crossed essentially matching the best known bounds. 
Our proof relies on the crossing number inequality and several probabilistic tools such as concentration of measure and the Lov\'asz local lemma. 
\end{abstract}

\section{Introduction}

A \emph{drawing} of a graph $G$ is a mapping, in which every vertex of $G$ is mapped to a distinct point in the plane, and every edge into a continuous curve connecting the images of
its endpoints. 
As is standard (see, e.g., \cite{PACH-k-planar}), we will assume that (1) no curve contains the image of any vertex other than its endpoints, (2) no two curves share infinitely many points, (3) no two curves are tangent to each other, and (4) no three curves pass through the same point. 
A \emph{crossing} in such a drawing is
a point where the images of two edges intersect, and the \emph{crossing number} of a graph $G$, denoted by $\CR(G)$, is the
smallest number of crossings achievable by any drawing of $G$ in the plane. 

The study of crossing numbers dates back to Paul Tur\'an's Brick Factory Problem \cite{turan1977note}. 
While working in a forced labor camp during the Second World War, Tur\'an wondered how to design an `efficient' rail system from the `kilns' to the `storage yards', where each kiln was to be connected by a direct track to each storage yard; the objective was to minimize the number of crossings, where the cars tended to fall off the tracks, requiring workers to reload the bricks onto the car. In the terminology introduced in the previous paragraph, this is precisely the problem of finding a drawing of the complete bipartite graph attaining its crossing number. Over the years, the crossing number has emerged as a central object of interest in discrete mathematics. We refer the reader to the recent book of Schaefer \cite{schaefer2018crossing} for a modern and thorough account of this area.   

The study of drawings of graphs with additional `local' restrictions on crossings has also attracted considerable attention in recent decades. As a natural relaxation of the standard notion of planarity, Ringel \cite{ringel1965sechsfarbenproblem} defined a graph to be \emph{1-planar} if it admits a drawing with at most one crossing on each edge.
Similarly, one can define \emph{k-planar} graphs for all integers $k \geq 1$ (we caution the reader that the notion of `$k$-planarity' in the previous sentence is completely different from the $k$-planar local crossing number that we will introduce later). Ringel was interested in a generalization of the $4$-color theorem to $1$-planar graphs, which would imply results for the problem of simultaneously coloring vertices and faces of planar graphs. $k$-planarity has emerged as one of the most widely studied generalizations of planarity, and has found applications in graph theory, graph algorithms, graph drawing and computational geometry (see, e.g., the annotated bibliography \cite{kobourov2017annotated}).
On the other hand, for applications of a similar nature to the one Tur\'an was interested in, it is more convenient to turn the above definition around, and define the \emph{local crossing number} of a graph $G$, denoted by $\LCR(G)$, to be the minimum $k$ for which the graph is $k$-planar. 
In other words, $\LCR(G)$ is the smallest integer $L$ for which there exists a drawing of $G$ such that there are at most $L$ crossings along any edge.     

Motivated by applications to the design of printed and integrated circuits, Owens \cite{owens1971biplanar} defined the \emph{biplanar crossing number}, denoted by $\CR_2(G)$, to be the minimum sum of the crossing numbers of two graphs $G_0$ and $G_1$  (on the same vertex set as $G$), whose union is $G$. 
This was extended by Shahrokhi et al. \cite{shahrokhi2007k} to \emph{$k$-planar crossing numbers}, denoted by $\CR_k(G)$, for all integers $k\geq 1$ in the natural way: for any graph $G$, $\CR_k(G)$ is the minimum of $\CR(G_1) + \CR(G_2) + \dots + \CR(G_{k})$, where the minimum is taken over all graphs $G_1, G_2,\dots G_{k}$ on the same vertex set with $G$ such that $\cup_{i=1}^{k}E(G_i) = E(G)$. 
We remark that while the preceding definitions are natural extensions of the definition of the crossing number, $\CR_k(G)$ for $k \geq 2$ seems to behave quite differently from $\CR(G)$; for instance, it is well known that testing whether $\CR(G)=0$ can be done in linear time (\cite{boyer2004cutting}), whereas testing whether $\CR_2(G)=0$ is already NP-complete (\cite{cabello2013adding}). 
For a detailed introduction to $k$-planar crossing numbers, we refer the reader to survey papers \cite{czabarka2008biplanar-survey-II} and \cite{czabarka2006biplanar-survey-I}. 

Recently, several researchers have investigated the relationship between $\CR(G)$ and $\CR_k(G)$. Czabarka, S{\`y}kora, Sz{\'e}kely, and Vrt'o 
\cite{czabarka2008biplanar-survey-II} proved that for every graph $G$, 
$$\CR_2(G)\leq \frac{3}{8}\CR(G).$$
They also showed that this inequality does not remain true if the constant $3/8$ is replaced by anything smaller than $8/119$. This result was refined and extended to $\CR_k(G)$ for all $k\geq 1$ by Pach, Sz\'ekely, T\'oth, and T\'oth \cite{PACH-k-planar}. 

\begin{thm}\label{Pach-k-planar-thm}(Pach, Sz\'ekely, T\'oth, and T\'oth \cite{PACH-k-planar})
For every integer $k\geq 1$, 
$$\frac{1}{k^2}\leq \sup \frac{\CR_k(G)}{\CR(G)}\leq \frac{2}{k^2}-\frac{1}{k^3},$$
where the supremum is taken over all \emph{nonplanar} graphs $G$.
\end{thm}
 
Having in mind applications where any edge being crossed too many times constitutes a `bottleneck', we introduce in this paper the following notion of \emph{$k$-planar local crossing number}.
\begin{definition}
Let $k$ be a positive integer. For any graph $G$,  its \emph{$k$-planar local crossing number}, denoted by $\LCR_k(G)$, is the minimum of $\max\{\LCR(G_1),\LCR(G_2),\dots , \LCR(G_{k})\}$, where the minimum is taken over all graphs $G_1, G_2,\dots G_{k}$ such that $\cup_{i=1}^{k}E(G_i) = E(G)$.
\end{definition}

Our main results have a similar flavor as \cref{Pach-k-planar-thm}, and relate the local crossing number of a graph to its $k$-planar variant. Before stating them, we need to introduce some notation. For any graph $G$, let $e(G)$ denote the number of edges and let $v(G)$ denote the number of vertices. For each $\alpha>0$ and $\beta>0$, let $\g_{\alpha, \beta}$ denote the set of all graphs for which the maximum degree $\Delta(G)$ is no more than $\alpha$ times the average degree (i.e. $\Delta (G)\leq  2\alpha e(G)/v(G)$) and its local crossing number is at least $\beta$.

\begin{thm}\label{thm 4/9}
Fix $k \in \N$. For any $0<\varepsilon<1/10$,  
$$\sup_{G\in \g_{\alpha,\beta}}\frac{\LCR_k(G)}{\LCR(G)}\leq \frac{2}{k^2}+ \varepsilon,$$
provided $\beta \geq 1000\alpha^{2}\varepsilon^{-4}\left(\log{\alpha}+\log(1/\varepsilon)\right)^{2}$. 
Moreover, for $k=2$, the term $2/k^{2} = 1/2$ can be replaced by $4/9$.

On the other hand, even for the family of complete graphs $K_n$ (which, for all sufficiently large $n$, are in $\g_{\alpha,\beta}$ for any $\alpha \geq 1$ and $\beta \leq {n^2}/75$), we have
$$ \liminf_{n\to \infty}\frac{\LCR_k(K_n)}{\LCR(K_n)} \geq \frac{9}{58k^2}.$$ 
\end{thm}

For arbitrarily irregular graphs with sufficiently large local crossing number, we can instead prove the following weaker result. 
\begin{thm}\label{thm 1/k} 
Fix $k \in \N$. For any $0 < \varepsilon < 1/10$ and any graph $G$ with $\LCR(G) \geq 10\log(1/\varepsilon)/\varepsilon^{2}$,
$$\LCR_k(G)\leq \left(\frac{1}{k}+\varepsilon \right) \LCR(G).$$
\end{thm}

As our final result, we show that under similar conditions as in \cref{thm 4/9}, there is a way to partition the graph into $k$ planes such that \emph{both} the total number of crossings and the maximum of the local crossings are as small as the best known upper bounds \cref{Pach-k-planar-thm} and \cref{thm 4/9} (i.e. one may obtain the desirable features of both these theorems simultaneously which may be useful for applications)

\begin{thm}\label{combine cr and lcr}
Fix $k \geq 2$. For any $0<\varepsilon < 1/10$ and any $G\in \g_{\alpha,\beta}$ with average degree $d$, we can find a decomposition $G=G_1\cup\dots\cup G_k$ such that both of the following hold:
\begin{equation}
\sum_{i=1}^k \CR(G_i)\leq \left(\frac{2}{k^2}-\frac{1}{k^3}+\varepsilon\right) \CR(G),
\end{equation}
and
\begin{equation}
\max_{i\in [k]} \LCR(G_i)\leq \left(\frac{2}{k^2}+\varepsilon\right)\LCR(G),
\end{equation}
provided $\beta \geq 1000\alpha^{2}\varepsilon^{-4}(\log(\alpha)+\log(1/\varepsilon))^{2}$ and $\alpha < \varepsilon^{2}d^{4}/2000$. 
\end{thm}

Our proofs of \cref{thm 4/9} and \cref{combine cr and lcr} are based on the natural idea in \cite{czabarka2008biplanar-survey-II,czabarka2006biplanar-survey-I,PACH-k-planar} of randomly partitioning the vertices and assigning edges to planes based on which parts of the partition their endpoints land in. However, while the analysis in these papers requires only a calculation of the probability of the event that a given crossing in a fixed drawing `survives' this process, we additionally require tight concentration on the upper tail since we want to control the maximum of a collection of many dependent random variables. Interestingly, the famous crossing number inequality will play a crucial role in our proofs of concentration. In order to weaken the hypotheses under which our main results hold, we also use the Lov\'asz local lemma instead of a simple union bound in various places. Even so, we can only prove our results under certain restrictions on the graph. It would be interesting, in our opinion, to investigate whether similar results hold more generally. Another natural open problem is to close the gap between the lower and upper bounds in \cref{thm 4/9}. 

The remainder of the paper is organized as follows.  In Section 2 we will gather various preliminaries.  In Sections 3, 4, and 5, we will prove  \cref{thm 1/k}, \cref{thm 4/9}, and \cref{combine cr and lcr}, respectively.

\section{Tools and auxiliary results}
In this section, we have collected a number of tools and auxiliary results to be used in proving our main results.

\subsection{Probabilistic tools}
We will make use of the following well-known concentration inequality for sums of independent random variables due to Hoeffding \cite{hoeffding1963probability}.
\begin{lemma}[Hoeffding's inequality]
\label{Hoeffding}
Let $X_{1},\dots,X_{n}$ be independent random variables such that
$a_{i}\leq X_{i}\leq b_{i}$ with probability one. If $S_{n}=\sum_{i=1}^{n}X_{i}$,
then for all $t>0$,
\[
\Pr\left(S_{n}-\E[S_{n}]\geq t\right)\leq\exp\left(-\frac{2t^{2}}{\sum_{i=1}^{n}(b_{i}-a_{i})^{2}}\right)
\]
and 
\[
\Pr\left(S_{n}-\E[S_{n}]\leq-t\right)\leq\exp\left(-\frac{2t^{2}}{\sum_{i=1}^{n}(b_{i}-a_{i})^{2}}\right).
\]
\end{lemma}

We will also need a generalization of this inequality from sums to functions of bounded differences due to McDiarmid \cite{mcdiarmid1989method}.  
\begin{definition} 
\label{defn:bounded-differences}
Let $\mathcal{X}$ be an arbitrary set, and consider a function $g\colon \mathcal{X}^{n}\to \R$. We say that $g$ has \emph{bounded differences} if there exist nonnegative numbers $c_1,\dots, c_n$ such that 
$$\sup_{x_1,\dots,x_n,x_i' \in \mathcal{X}}|g(x_1,\dots,x_n)-g(x_1,\dots,x_{i-1},x_i',x_{i+1},\dots,x_{n})| \leq c_i $$
for all $i=1,\dots,n$. 
\end{definition}

\begin{lemma}[McDiarmid's inequality] 
\label{lemma:McDiarmid's-inequality}
Let $\mathcal{X}$ be an arbitrary set and let $(X_1,\dots,X_n)\in \mathcal{X}^n$ be an $n$-tuple of independent $\mathcal{X}$-valued random variables. Let $g\colon \mathcal{X}^n \to \R$ be a function with bounded differences, as in {\normalfont\cref{defn:bounded-differences}}, and let $\mu :=\E[g(X_1,\dots,X_n)]$. Then, for all $t > 0$, we have:
\begin{itemize}
\item $\Pr[g(X_1,\dots,X_n) \leq \mu - t ] \leq \exp\left(-\frac{2t^2}{\sum_{i=1}^{n} c_i^2} \right),$ and
\item $\Pr[g(X_1,\dots,X_n) \geq \mu + t] \leq \exp\left(-\frac{2t^2}{\sum_{i=1}^{n} c_i^2} \right).$
\end{itemize}
\end{lemma}

In what follows, we will usually not have strong enough concentration to take union bounds over collections of `bad' events. However, since the dependencies among our events will be limited, we can circumvent this obstacle by using instead the  Lov\'asz local lemma in its symmetric version (see, e.g., \cite{alon2004probabilistic}). 
Before stating it, we need the following
definition.

\begin{definition}
  Let $(A_i)_{i=1}^n$ be a collection of events in some probability
space. A graph $D$ on the vertex set $[n]$ is called a
\emph{dependency graph} for $(A_i)_{i=1}^n$ if $A_i$ is mutually
independent of all the events $\{A_j: ij\notin E(D)\}$.
\end{definition}

\begin{lemma}[Lov\'asz local lemma]
Let $(A_i)_{i=1}^n$ be a sequence of $n$ events in some probability
space and let $D$ be a dependency graph for $(A_i)_{i=1}^{n}$. Let $\Delta:=\Delta(D)$ be the maximum degree of this dependency graph, and
suppose that for every $i\in [n]:=\{1,2,\dots,n\}$ 
we have $\Pr (A_i)\leq q$.  
If $3q(\Delta+1)<1$, then $\Pr \left(\bigcap_{i=1}^{n} \overline{A_i}\right)>\left(1-1/(\Delta+1)\right)^n$.
\end{lemma}
\begin{remark}
The local lemma is typically stated with the constant $3$ replaced by $e$. As it will make no difference in our analysis, we prefer to use the slightly worse constant $3$ since we will later, as is customary, use $e$ to denote an edge in a graph.
\end{remark}

\subsection{Intersection graph}
\begin{definition}\label{def intersection graph}
For any drawing $D$ of a graph $G$, the \emph{intersection graph} or \emph{edge-crossing graph} of $G$ with respect to the drawing $D$, denoted by $\I(G)_{D}$, is defined to be the graph whose vertices correspond to edges of $G$, and two vertices are adjacent if and only if their corresponding edges in $G$ cross each other in the drawing $D$.
\end{definition}

\begin{rmk}
This should not be confused with the \emph{line graph} of $G$ (where the vertices correspond to edges of $G$ and two vertices are adjacent if and only if the corresponding edges share an endpoint in $G$) or the \emph{string graph} of a drawing of curves in the plane (where vertices correspond to curves and two vertices are adjacent if and only if the corresponding curves have non-empty intersection).
\end{rmk}

For the remainder of this subsection, fix any drawing $D$ of $G$ that has local crossing number equal to  $\LCR(G)$ and let $\I(G):= \I(G)_{D}$. Note that by definition of the local crossing number, $\Delta(\I(G)) = \LCR(G)$. Using $\I(G)$, we can immediately deduce some simple facts about the $k$-planar local crossing number. For instance, if $\LCR(G)=k$, then $\LCR_{k+1}(G)=0$. Indeed, by Brook's theorem \cite{brooks_1941}, the chromatic number of $\I(G)$ is at most $\Delta(\I(G))+1=k+1$, and we can use any such coloring of $\I(G)$ with $k+1$ colors to decompose $G$ into $k+1$ edge disjoint planar graphs. Another simple result is that if $\LCR(G)=2$, then $\LCR_2(G)\leq 1$; indeed, $\Delta(\I(G))=2$ implies that $\I(G)$ is a disjoint union of cycles, paths, and isolated vertices, and hence can be decomposed into $2$ parts, each of maximum degree at most $1$. 

\subsection{Crossing number inequality}
The crossing number inequality 
is an important tool in graph theory which shows that any drawing of a sufficiently dense graph has a large number of crossings. It has many applications, prominently in bounding the number of incidences between points and lines/curves in the plane (see, e.g., \cite{szekely1997crossing-application-ST}). Here, we state it with the presently best known constant, which is due to Ackerman \cite{ackerman2006maximum}; similar results with weaker constants appeared in \cite{leighton1984new,pach1997graphs, pach2006improving}.
\begin{thm}[Ackerman \cite{ackerman2006maximum}]\label{crossing number ineq} For any graph $G$ with $m$ edges and $n$ vertices such that $m>6.95n$, we have
$$\CR(G)\geq \frac{m^3}{29 n^2}.$$
\end{thm}
In this paper we will use the crossing number inequality via the following simple corollary which follows immediately by combining the crossing number inequality with the obvious inequality $\LCR(G) \geq  2\CR(G)/m$. 
\begin{cor}\label{lcr vs avg deg} 
Under the same assumption as in \cref{crossing number ineq}, $\LCR(G)\geq\frac{2m^2}{29n^2}$.
\end{cor}
\section{Proof of \cref{thm 1/k}}

Fix an arbitrary drawing of $G$ attaining $\LCR(G)$, and let $\I(G)$ be the intersection graph of $G$ with respect to this drawing. In particular, $\Delta(\I(G)) = \LCR(G)$. 
\cref{thm 1/k} follows immediately from the following proposition applied to $\I(G)$.
\begin{prop}
Fix $k \in \N$. Given $0 < \varepsilon < 1/10$ and a graph $H$ with $\Delta(H) \geq \beta(\varepsilon):=10\log(1/\varepsilon)/\varepsilon^2$, we can partition the vertex set $V(H)$ of $H$ into $k$ parts $V_1,\dots,V_k$ such that the maximum degree of each of the induced graphs $H[V_i]$  is at most $(k^{-1}+\varepsilon)\Delta(H)$.
\end{prop}

\begin{proof}
Let $V_1,\dots, V_k$ be a random partitioning of $V(H)$ generated by assigning independently to each vertex $v\in V(H)$ an element chosen uniformly at random from $[k]$. For each $i \in [k]$, let $H_i := H[V_i]$ denote the graph induced by $H$ on the vertex set $V_i$. We will show that 
with positive probability, $\deg_{H_i}(v) \leq (k^{-1}+\varepsilon)\Delta(H)$ for all $i \in [k]$ and $v \in V_i$ provided that $\Delta(H)\geq \beta(\varepsilon)$. For this, we will first use Hoeffding's inequality to upper bound the probability that a given vertex has degree larger than $k^{-1}\deg_H(v) + \varepsilon \Delta(H)$, and then use the local lemma to complete the proof. 

Accordingly, fix a vertex $v$ in $V(H)$ and let $N_H(v)$ denote the set of its neighbors in $H$. Without loss of generality, we may assume $v \in V_1$.   
Since each vertex $u$ in $N_H(v)$ is assigned to the same part as $v$ independently with probability $1/k$, it follows that $\E[\deg_{H_1}(v)] = \deg_H(v)/k$. Let $E_v$ denote the event that $\E[\deg_{H_1}(v)] > k^{-1}\deg_H(v) + \varepsilon \Delta(H)$. Then, by Hoeffding's inequality, we get that     
$$\Pr(E_v) \leq \exp\left(-\frac{2\varepsilon^2\Delta(H)^2}{\deg_H(v)^2}\right) \leq \exp\left(-2\varepsilon^{2}\Delta(H)\right). 
$$
For any two vertices $v_1$ and $v_2$ which neither share a common neighbor in $H$ nor are adjacent in $H$, the events $E_{v_1}$ and $E_{v_2}$ are independent since they depend on disjoint sets of vertices. As the maximum degree of any vertex in $H$ is $\Delta(H)$, it follows that any event $E_v$ can depend on at most $\Delta(H)^2$ other events $E_w$. Therefore, by the local lemma, we see that 
$$\Pr[\bigcap_{v\in V(H)} \overline{E_v}] >0$$
as long as $3(\Delta(H)^2+1)\exp(-2\varepsilon^2 \Delta(H)) < 1$, 
which is true if (say) $\Delta(H)\geq \beta(\varepsilon):= 10\log(1/\varepsilon)/\varepsilon^2$
\end{proof}

\section{Proof of \cref{thm 4/9}}
\subsection{Upper bound}
We will follow the same construction as in \cite{PACH-k-planar} although our choice of parameters and our analysis will be different. Throughout, we will work with a fixed drawing $D$ of $G$ that attains $\LCR(G)=:L$. Let $V_1,\dots, V_k$ be a random partitioning of $V(G)$ generated by assigning independently to each vertex $v\in V(G)$ an element $\xi_v$ chosen randomly from $[k]$ according to the distribution $\Pr[\xi_v=i] = p_i$, where $p_i \in [0,1]$ satisfying $p_1+\dots+p_k = 1$ are fixed constants which we will motivate and state explicitly in \cref{lemma:parameters}.  
We define the \emph{type} of an edge $(u,v) \in G$ to be the set $\{\xi_u,\xi_v\}$.

For a given collection of $\{\xi_u\}_{u\in V(G)}$ (equivalently, a given partition of the vertices into $V_1,\dots,V_k$), we obtain a decomposition of $G$ into $k$ planes as follows: for each $\ell \in [k]$, we take the $\ell^{th}$ plane $G_\ell$ to consist of edges between those $V_i$ and $V_j$ (where $i,j \in [k]$ are not necessarily distinct) for which $i+j\equiv \ell \mod k$. In other words, we take $G_\ell$ to consist of all edges of type $\{i,j\}$ where $i+j\equiv \ell \mod k$. 
Note that for any $\ell$ and for any $i$, there is a unique $j$ such that $G_\ell$ has an edge connecting a vertex in $V_i$ to a vertex in $V_j$. In particular, each connected component of $G_\ell$ consists of edges of the same type. Therefore, each $G_\ell$ may be drawn in a manner such that two edges cross only if they are of the same type and crossed in the original drawing $D$; indeed, we first draw the edges of $G_\ell$ according to the original drawing $D$. Next, we translate the connected components of $G_\ell$
sufficiently far from each other so that
no two edges of different types intersect, and such that no new crossings are introduced. 

We now upper bound the expected number of times a given edge $e=(u,v)$ is crossed at the end of our procedure, conditioned on the values of $\xi_u$ and $\xi_v$. Later, we will explain why we need to work with this more refined quantity, as opposed to just the expectation. As mentioned above, every crossing of $e$ arises from an edge $e'$ such that $e$ and $e'$ crossed in the original drawing $D$. 
If $\xi_u = \xi_v = i$ for some $i\in [k]$, then the probability that $e$ and $e'$ still cross after our procedure is $p_i^2$. 
On the other hand, if $\xi_u = i$ and $\xi_v = j$ for some $i,j\in [k]$ with $i\neq j$, 
then $e$ and $e'$ continue to cross after our procedure if and only if one vertex of $e'$ belongs to $V_i$ and the other vertex of $e'$ belongs to $V_j$. This happens with probability $2p_ip_j$. The following simple lemma provides the optimal choice of the parameters $p_i$ for our analysis. 
\begin{lemma} 
\label{lemma:parameters}
Let $\gamma_k := \min_{p_i\in[0,1], \sum_i p_i=1} \left\{\max_{i,j\in[k],i\neq j}\{p_i^2, 2p_ip_j\}\right\}$. Then, 
$$\gamma_k=\begin{cases}
\frac{2}{k^2} \quad \text{ for } k\geq 3, \\ 
\frac{4}{9}\quad \text{ for } k=2. 
\end{cases}$$
Moreover, $\gamma_2$ is attained by the choice $p_1 = {2/3}$ and $p_2 = 1/3$, whereas for $k\geq 3$, $\gamma_k$ is attained by the choice $p_i = 1/k$ for all $i\in[k]$. 
\end{lemma}
\begin{proof} Without loss of generality, we may assume that $p_1\geq p_2\geq\dots\geq p_k$. 
If $p_1=p_2$, the maximum of $\{p_i^2,2p_i,p_j\}$ is clearly $2p_1p_2= 2p_1^2$ which is at least $\frac{2}{k^2}$ since $p_1\geq \frac{1}{k}$. On the other hand, if $p_1>p_2$, then the maximum is $\max\{p_1^2,2p_1p_2\}$. This quantity is minimized when $p_1 = 2p_2$, in which case it is equal to $4p_2^2$. 
Moreover, $4p_2^2$ is minimized when $p_i = p_2$ for all $i\geq 2$, so that $p_1/2=p_2=\dots=p_k=1/(k+1)$. Thus, when $p_1 > p_2$, the minmax is $4/(k+1)^2$, which is only smaller than $2/k^2$ when $k=2$.
\end{proof}

Since the values of $p_i$ that we use in the random partitioning are those coming from \cref{lemma:parameters}, it follows from the paragraph preceding the statement of the lemma that the expected number of crossings of $e = (u,v)$, 
conditioned on the values of $\xi_u$ and $\xi_v$, is at most $\gamma_k L$.  
Let $g(e)$ denote the random variable recording the number of crossings of $e$ 
at the end of our procedure, and for $k_1,k_2\in[k]$, let $A_{e,k_1,k_2}$ denote the event that $g(e) > (\gamma_k + \varepsilon)L$ conditioned on $\xi_u = k_1$ and $\xi_v = k_2$. We now use McDiarmid's inequality to show that $A_{e,k_1,k_2}$ occurs with sufficiently low probability.

Let $e_1,\dots,e_s$ be an enumeration of the edges which cross $e=(u,v)$ in the drawing $D$ of $G$ and let $v_1,\dots,v_t$ denote an enumeration of their endpoints \emph{other than $u$ and $v$}. 
For each $i \in [t]$, 
let $c_i$ denote the number of edges incident to $v_i$ that cross $e$ in the drawing $D$ of $G$. Then, conditioned on the values of $\xi_u$ and $\xi_v$, $g(e)$ 
depends only on the independent random variables variables $\xi_{v_1},\dots, \xi_{v_t}$, and moreover, for all $i\in[t]$, changing the value of $\xi_{v_i}$ can change the value of $g(e)$ by at most $c_i$. Therefore, by McDiarmid's inequality, we get that   
\begin{equation}
\label{eqn:mcd}
\Pr(A_{e,k_1,k_2})\leq \exp\left(-\frac{2(\varepsilon L)^2}{c_1^2+\dots+c_t^2}\right) \leq \exp\left(-\frac{2\varepsilon^2 L}{\Delta}\right),
\end{equation} 
where in the second inequality, we have used that  
$\sum_i c_i^2 \leq \Delta \sum_i c_i \leq \Delta L$. We remark that the same argument does not work if we do not condition on the values of $\xi_u$ and $\xi_v$, since in this case $g(e)$ also depends on the random variables $\xi_u$ and $\xi_v$ which can influence the value of $g(e)$ by up to $L$. Now, from the definition of $A_{e,k_1,k_2}$ and \cref{eqn:mcd}, it immediately follows using the law of total probability that if we let $A_e$ be the event that $g(e) > (\gamma_k + \varepsilon)L$, then $\Pr(A_e) \leq \exp(-2\varepsilon^2 L /\Delta)$.  

We wish to use the local lemma to show that with positive probability, none of the events $A_e$ occurs. For this, we begin by observing that if two edges $e_1$ and $e_2$ are such that $e_1\cap e_2 = \varnothing$ and there is no triple $(v,e'_1,e'_2)$ for which $v\in e'_1 \cap e'_2$, $e'_1$ crosses $e_1$ and $e'_2$ crosses $e_2$, then the events $A_{e_1}$ and $A_{e_2}$ depend on disjoint collections of random variables.  
It follows that each event $A_e$ can depend on at most $2L^2\Delta$+ $2\Delta$ other $A_{e'}$'s; there are at most $2\Delta$ edges incident to $e$, and there are at most $2L^2\Delta$ triples $(v,e'_1,e'_2)$ as above since there are at most $L$ choices for $e_1'$, $2$ choices for $v\in e_1'$, at most $\Delta$ choices for $e_2' \ni v$ and at most $L$ choices for $e_2$ crossing $e_2'$.  
In particular, we can apply the local lemma provided $3 \Pr(A_e) \left(2L^2\Delta + 2\Delta +1\right) <1$. 

Note that until now, we have not used our assumption on the structure of $G$. This assumption will be used in the current paragraph to show that $3 \Pr(A_e) \left(2L^2\Delta + 2\Delta +1\right) <1$. First, since $L \geq 2m^{2}/29n^{2}$ by \cref{lcr vs avg deg}, and since $\Delta < 2\alpha m/n$ by assumption, it follows that $\Delta < \alpha\sqrt{58L}$. Therefore, for $L \geq 1$,
\begin{align*}
3 \Pr(A_e) \left(2L^2\Delta + 2\Delta +1\right) &\leq 15 L^2\Delta\Pr(A_e)\\
&\leq 120\alpha L^{5/2}\exp(-2\varepsilon^2 L /\Delta)\\
&\leq 120\alpha L^{5/2}\exp\left(-\varepsilon^{2} \sqrt{L}/4\alpha\right),
\end{align*}
and note that the right hand side is less than $1$ if $L \geq \beta$ with $\beta$ as in the statement of the theorem.

\begin{remark}
\label{rmk:prob-bound}
The above proof shows that, in fact, $\Pr[\cap_e\overline{A_e}] \geq (1-1/(2L^2\Delta+2\Delta+1))^{m}$. We will use this in the proof of \cref{combine cr and lcr}.
\end{remark}


\subsection{Lower bound} 
It is well-known (see \cite{schaefer2018crossing}) that 
$$\frac{1}{30}\leq \lim_{n\to \infty} \frac{\LCR(K_n)}{{n\choose 2}}\leq \frac{2}{9}.$$
On the other hand, when we decompose $K_n$ into $k$ subgraphs, there must exist one with at least $\frac{1}{k}{n\choose 2}$ edges. Without loss of generality, we may assume that $G_1$ is such a subgraph. Then,
$$\LCR_k(K_n) \geq \LCR(G_1) \geq \frac{2e(G_1)^{2}}{29n^2} \geq \frac{\frac{2}{k^2}{n\choose 2}^{2}}{29n^{2}},$$
where the second inequality follows from \cref{lcr vs avg deg}. 
Hence 
$$\liminf_{n\to \infty}\frac{\LCR_k(K_n)}{\LCR(K_n)}\geq \liminf_{n\to \infty}\frac{\frac{2}{k^{2}} {n\choose 2}^{2}}{29 n^2}\left(\frac{2}{9}{n\choose 2}\right)^{-1}\geq \frac{9}{58k^2}.$$

\section{Proof of \cref{combine cr and lcr}}
Let $G$ have $n$ vertices, $m$ edges, maximum degree $\Delta$, crossing number $\CR(G)=C$, and local crossing number $\LCR(G)=L$. Consider the random vertex partitioning and associated $k$-planar decomposition of $G$ into $G_1,\dots,G_k$ as in the proof of \cref{thm 4/9}, and consider drawings of $G_1,\dots,G_k$ in the plane as before. For each $i \in [k]$ and for this choice of drawing of $G_i$, we will use $C_i$ to denote the total number of crossings, and $L_i$ to denote the maximum number of times any edge is crossed.  
Let $E_1$ denote the event that $\sum_{i=1}^k C_i\leq \left(\frac{2}{k^2}-\frac{1}{k^3}+\varepsilon\right) \CR(G)$, and let $E_2$ denote the event that $\max_{i\in [k]} L_i\leq \left(\frac{2}{k^2}+\varepsilon\right)\LCR(G)$. In order to complete the proof, it suffices to show that $\Pr(E_1\cap E_2)>0$, or equivalently, that $\Pr(\overline{E_1}\cup \overline{E_2}) < 1$. We will do so by showing that $\Pr(\overline{E_1}) < \Pr(E_2)$. 

We begin by upper bounding $\Pr(\overline{E_1})$. Let $X:=C_1+\dots+C_k$. The analysis in \cite{PACH-k-planar} shows that $\mathbb{E}[X]=\left(\frac{2}{k^2}-\frac{1}{k^3}\right) \CR(G)$, where the expectation is taken with respect to the choice of the vertex partitioning $\{\xi_{v}\}_{v\in V(G)}$. Observe that changing the value of any $\xi_{v}$ can change $X$ by at most the total number of crossings in the original drawing of $G$ that $v$ is involved in; by assumption, this is at most $\Delta L$, since $v$ is incident to at most $\Delta$ edges, each of which is crossed at most $L$ times. Therefore, by McDiarmid's inequality,
$$\Pr(\overline{E_1})=\Pr(X>\mathbb{E}[X]+\varepsilon C)\leq \exp{\left(-\frac{2(\varepsilon C)^2}{n(L\Delta)^2}\right)}.$$

Moreover, by \cref{rmk:prob-bound} 
$$\Pr(E_2) \geq \left(1-\frac{1}{2L^2\Delta+2\Delta+1}\right)^m\geq \left(1-\frac{1}{2L^2\Delta}\right)^{m}\geq \exp{\left(-\frac{m}{L^2\Delta}\right)},$$
where the final inequality holds since $L \geq 1000$ by assumption. Therefore, in order to show that $\Pr(\overline{E_1}) < \Pr(E_2)$, we simply need to show that 
$$\frac{m}{L^{2}\Delta} < \frac{2(\varepsilon C)^{2}}{n (L\Delta)^2},$$
which is equivalent to $C^{2} > mn\Delta/2\varepsilon^{2}$.
By the crossing number inequality $C \geq m^{3}/29n^{2}$, this is implied by $m^{6}/29n^{4} > mn\Delta / 2\varepsilon^{2}$, which follows immediately from our assumption that $\alpha <\varepsilon^{2}(m/n)^{4}/1000$, and hence $\Delta < \varepsilon^{2} (m/n)^{5}/500$.


\section*{Acknowledgements}
This material is based upon a question raised at the Mathematics Research Communities workshop ``Beyond Planarity: Crossing Numbers of Graphs'', organized by the American Mathematical Society, with the 
support of the National Science Foundation under Grant Number DMS 1641020, which was attended by the first three authors. The first three authors would like to thank  Laszlo Sz\'ekely, Libby Taylor and Zhiyu Wang
for helpful conversations.


\bibliographystyle{plain}
\bibliography{local-crossing}
\end{document}